\documentclass[journal,twoside,web]{IEEEcolor}
\usepackage{generic}
\usepackage{cite}
\usepackage{amsmath,amssymb,amsfonts}
\usepackage{algorithmic}
\usepackage{graphicx}
\usepackage{textcomp}
\usepackage{bm}
\usepackage{caption,subcaption}
\usepackage{comment}
\usepackage{xcolor}
\usepackage{epstopdf}

\let\labelindent\relax
\usepackage{enumitem}
\newtheorem{theorem}{Theorem}   
\newtheorem{lemma}{Lemma}

\newtheorem{corollary}{Corollary}
\newtheorem{example}{Example}
\newtheorem{remark}{Remark}

\newtheorem{definition}{Definition}

\def\BibTeX{{\rm B\kern-.05em{\sc i\kern-.025em b}\kern-.08em
    T\kern-.1667em\lower.7ex\hbox{E}\kern-.125emX}}
\begin{document}
\title{Necessary and Sufficient Conditions for Data-driven Model Reference Control}
\author{Jiwei Wang, Simone Baldi,~\IEEEmembership{Senior Member,~IEEE} and Henk J. van Waarde,~\IEEEmembership{Member,~IEEE} \vspace{-2em}
\thanks{This work was partially supported by the Jiangsu Research Center of Applied Mathematics grant BK20233002, by the China Scholarship Council 202206090231, and by the Natural Science Foundation of China grant 62073074, 62233004. (Corresponding author: Simone Baldi)}
\thanks{J. Wang is with School of Cyber Science and Engineering, Southeast University, China, and with the Bernoulli Institute for Mathematics, Computer Science and Artificial Intelligence, University of Groningen, The Netherlands (e-mail: jiwei.wang@rug.nl)}
\thanks{S. Baldi is with School of Mathematics, Southeast University, China (email: simonebaldi@seu.edu.cn).}
\thanks{H. J. van Waarde is with the Bernoulli Institute for Mathematics, Computer Science and Artificial Intelligence, University of Groningen, The Netherlands (e-mail: h.j.van.waarde@rug.nl).}}

\maketitle

\begin{abstract}
	The objective of model reference control is to design a controller that regulates the system's behavior so as to match a specified reference model.
	This paper investigates necessary and sufficient conditions for model reference control from a data-driven perspective, when only a set of data generated by the system is utilized to directly accomplish the matching.
	Noiseless and noisy data settings are both considered.
	Notably, all methods we propose build on the concept of data informativity and do not rely on persistently exciting data.
	%Furthermore, we take the stability of systems satisfying approximate model matching into account, and establish necessary and sufficient conditions for designing a model reference controller with guaranteed stability from data.
\end{abstract}

\begin{IEEEkeywords}
Data-driven control, model reference control, data informativity, quadratic matrix inequalities.
\end{IEEEkeywords}

\hyphenation{MRC}

\section{Introduction}

In many applications of control, spanning from flight to process control, desirable performance is specified in terms of reference dynamics.
The goal is to find a controller such that the closed-loop dynamics matches the given reference model.
Such an architecture is often referred to as model reference control (MRC).
%In most studies of this field, experimental data generated by the dynamical system are available.
In some MRC scenarios with system uncertainty, the data generated by the system are used online to update the controller parameters, leading to the study of model reference adaptive control (MRAC) \cite{ioannou2006adaptive}.
In classical MRAC theory, a persistency of excitation condition is imposed on the system data to guarantee the convergence of the parameters to their ideal values.
However, persistency of excitation is known to be conservative and techniques such as combined/composite MRAC \cite{lavretsky2009combined} and concurrent learning \cite{chowdhary2010concurrent}, have been proposed to relax this condition, leading to weaker conditions on the data, such as initial excitation \cite{roy2017combined} and finite excitation \cite{lee2019concurrent}.

When only a set of data collected offline from the system can be utilized for control, we obtain scenarios in line with the theory of offline data-driven control, aiming to map the collected data to the controller gains directly \cite{hou2013model}.
In this area, a result that plays an important role is the \textit{fundamental lemma} by Willems et al. \cite{willems2005note}, %stating that all system trajectories can be represented by a finite set of system trajectories whose input is persistently exciting.
which provides a data-dependent representation of the system, using persistently exciting data.
Apart from MRC\cite{breschi2021direct}, the fundamental lemma leads to plenty of other results, such as robust control \cite{de2019formulas}, stochastic optimal control\cite{pan2022stochastic}, and model predictive control\cite{schmitz2022willems}.
Recently, \cite{van2020data} proposed a new framework to study the informativity of data and showed that persistency of excitation is not necessary in certain data-driven analysis and control problems. 
To apply the informativity framework to noisy data, \cite{van2020noisy} extended the S-lemma \cite{polik2007survey} to matrix S-lemma, further studied in \cite{van2023quadratic}.
We refer to \cite{van2023informativity} for more studies on the informativity framework for data-driven analysis and control.

This paper studies data-driven MRC in both the settings of noiseless and noisy data. 
Sufficient conditions for data-driven MRC have been studied in \cite{breschi2021direct}, where the design is based on persistently exciting data.
Our aim is to develop \emph{necessary and sufficient} conditions under which the data contain enough information for the design of a model reference controller. 
In case that the data are informative, we also want to develop methods for finding suitable controllers from the given data. 
We will follow the data informativity framework, utilizing and extending results from \cite{van2023quadratic} and \cite{van2022data}.
%Data-driven MRC has been studied in \cite{breschi2021direct}, where the design is based on persistently exciting data;
%\cite{breschi2021direct} also gives sufficient conditions to handle independent and identically distributed noise with zero-mean. 
%Compared with the existing work, 
The main contributions of this paper are the following:
\begin{itemize}
	\item We define the notion of informative data for MRC. 
	A set of data are informative for MRC if there exists a controller that achieves model matching for \emph{all} systems consistent with the data.
	We provide a necessary and sufficient condition for this notion of informative data in Theorem \ref{TDIMRC} and provide a method for finding a model reference controller via data-based linear matrix inequalities (LMIs) in Corollary \ref{CDI}.
	\item To address the noisy scenario, we assume that the noise satisfies a general	quadratic matrix inequality  (QMI) and we define the notion of informative data for approximate MRC. 
	A set of data are informative for approximate MRC if there exists a controller such that the distance of any consistent closed-loop system to the reference model is less than a prescribed threshold.
	We provide a necessary and sufficient condition for this notion of informative data in Theorem \ref{TDIAMRC} using the matrix S-lemma and matrix Finsler's lemma.
	\item We find a necessary and sufficient condition to guarantee the stability of \emph{all} systems satisfying approximate matching conditions (cf. Theorem \ref{TI}).
	To avoid the cases where, although approximate MRC is achieved, the closed-loop system may not be stable, we obtain a necessary and sufficient condition for a set of data to achieve approximate MRC with guaranteed stability in Theorem \ref{TDIAMRCSG}.
\end{itemize}

This paper is structured as follows. 
Section II recalls preliminary results.
In Section III we study informativity for MRC in the noiseless setting, while Section IV resolves the problem of approximate MRC in the noisy case. 
These results are further studied in Section V to guarantee the stability of the closed-loop system.
In Section VI, a numerical example of a highly maneuverable aircraft model is provided.
Section VII concludes this paper.

\section{Preliminaries}

\subsection{Notations}
We denote the field of real numbers by $ \mathbb{R} $, the space of real $ n $-dimensional vectors by $ \mathbb{R}^n $, the space of real $ n \times m $ matrices by $ \mathbb{R}^{n\times m} $ and the space of all real \textit{symmetric} $ n \times n $ matrices by $ \mathbb{S}^{n} $.
Given any matrix $ A $, its \textit{transpose} is denoted by $ A^{\top} $ and its \textit{Moore-Penrose pseudo-inverse} is denoted by $ A^{\dagger} $.
A matrix $ A \in \mathbb{S}^{n} $ is said to be \textit{positive definite} (denoted by $ A>0 $) if $ x^{\top}Ax>0 $ for any nonzero $ x\in\mathbb{R}^n $, and \textit{positive semidefinite} (denoted by $ A\geq0 $) if $ x^{\top}Ax\geq0 $ for any $ x\in\mathbb{R}^n $.
\textit{Negative definite} and \textit{negative semidefinite} matrices are defined analogously and denoted by $ A<0 $ and $ A\leq0 $, respectively.
We write $ A\not\leq0 $ if $ A $ is not negative semidefinite.
A square matrix is \textit{Schur} if all its eigenvalues have modulus strictly less than $ 1 $.
%A map $ x:\mathbb{R} \to \mathbb{R}^p $ is \textit{bounded} if there exists a $ \delta>0 $ such that $ |x(t)| \leq \delta\ \forall t \geq 0 $.

\subsection{Quadratic matrix inequalities}

Define the sets
\begin{align*}
\mathcal{Z}_{r}(\Pi):=&\left\{Z\in \mathbb{R}^{r \times q}\Bigg|\begin{bmatrix}
I_q\\Z
\end{bmatrix}^{\top}\Pi
\begin{bmatrix}
I_q\\Z
\end{bmatrix}\geq0\right\},\\
\mathcal{Z}^0_{r}(\Pi):=&\left\{Z\in \mathbb{R}^{r \times q}\Bigg|\begin{bmatrix}
I_q\\Z
\end{bmatrix}^{\top}\Pi
\begin{bmatrix}
I_q\\Z
\end{bmatrix}=0\right\},\\
\text{and }\qquad\mathcal{Z}^+_{r}(\Pi):=&\left\{Z\in \mathbb{R}^{r \times q}\Bigg|\begin{bmatrix}
I_q\\Z
\end{bmatrix}^{\top}\Pi
\begin{bmatrix}
I_q\\Z
\end{bmatrix}>0\right\},\qquad\qquad\qquad
\end{align*}
where \begin{equation}\label{pi}
\Pi=\begin{bmatrix}
\Pi_{11}&\Pi_{12}\\
\Pi_{21}&\Pi_{22}
\end{bmatrix}\in \mathbb{S}^{q+r}.
\end{equation}
%Partition $ \Pi $ as $  $, where $ \Pi_{11} \in \mathbb{S}^{q}, \Pi_{22} \in \mathbb{S}^{r} $.
Denote the (generalized) Schur complement of $\Pi$ with respect to $\Pi_{22}$ as $ \Pi|\Pi_{22}=\Pi_{11}-\Pi_{12}\Pi^{\dagger}_{22}\Pi_{21} $.
Consider the case that $ \Pi_{22}\leq 0 $ and $ \ker\Pi_{22}\subseteq\ker\Pi_{12} $, implying that 
%$\Pi_{12}\Pi^{\dagger}_{22}\Pi_{22}\Pi^{\dagger}_{22}\Pi_{21}=\Pi_{12}\Pi^{\dagger}_{22}\Pi_{21} $, indicating that 
$ \forall Z \in \mathbb{R}^{r \times q}, $
\begin{equation}\label{Pi}
\begin{bmatrix}
I_q\\Z
\end{bmatrix}^{\top}\Pi
\begin{bmatrix}
I_q\\Z
\end{bmatrix}=\Pi|\Pi_{22}+(Z+\Pi_{22}^{\dagger}\Pi_{21})^{\top}\Pi_{22}(Z+\Pi_{22}^{\dagger}\Pi_{21}).
\end{equation}
Under the above two conditions on $ \Pi_{22} $, we have that $ \mathcal{Z}_{r}(\Pi) \neq \emptyset $ if and only if $ \Pi|\Pi_{22}\geq 0 $.
Hence, we define the set
\begin{equation*}
\mathbf{\Pi}_{q,r}=\left\{\Pi\in\mathbb{S}^{q+r}\big|\Pi_{22}\leq 0, \ker\Pi_{22}\subseteq\ker\Pi_{12}, \Pi|\Pi_{22}\geq0\right\},
\end{equation*}
and study the set $ \mathcal{Z}_{r}(\Pi) $ with $ \Pi \in \mathbf{\Pi}_{q,r} $ \cite{van2023quadratic}.
These notions are instrumental to recall three results about QMIs.
\begin{lemma}\label{LZ}
	($ \!\!\! $\cite[Lemma 4.5]{van2023quadratic}).
	Let $ N\in \mathbf{\Pi}_{q,r} $. 
	Let $ x\in \mathbb{R}^q $ and $ y\in \mathbb{R}^r $ be vectors, with $ x $ nonzero, such that $ \begin{bmatrix}
	x\\y
	\end{bmatrix}^\top N \begin{bmatrix}
	x\\y
	\end{bmatrix}\geq0. $
	Then, 
	$
	\exists Z \in \mathcal{Z}_{r}(N) $, s.t. $ \begin{bmatrix}
	x\\y
	\end{bmatrix}= \begin{bmatrix}
	I\\Z
	\end{bmatrix}x.
	$
\end{lemma}
\vspace{0.3em}
\begin{lemma}\label{MSL}
	(Matrix S-lemma, \cite[Thm. 4.7]{van2023quadratic}). Let $ M,N\in \mathbb{S}^{q+r} $.
	Assume that $ N\in \mathbf{\Pi}_{q,r} $ and $ N \not\leq 0 $.
	Then, $ \mathcal{Z}_r(N)\subseteq \mathcal{Z}_r(M) $ if and only if 
	$
	\exists \alpha\geq0, \text{ s.t. } M-\alpha N\geq 0.
	$
\end{lemma}
\begin{lemma}\label{MFL}
	(Matrix Finsler's lemma, \cite[Thm. 4.8]{van2023quadratic}). 
	Let $ M=\begin{bmatrix}
	M_{11}&M_{12}\\M_{21}&M_{22}
	\end{bmatrix} \in \mathbb{S}^{q+r} $ and $ N=\begin{bmatrix}
	N_{11}&N_{12}\\N_{21}&N_{22}
	\end{bmatrix} \in \mathbb{S}^{q+r} $.
	Define $ \Theta:=\begin{bmatrix}
	I\\-N_{22}^{\dagger}N_{21}
	\end{bmatrix}^{\top}M
	\begin{bmatrix}
	I\\-N_{22}^{\dagger}N_{21}
	\end{bmatrix}\in \mathbb{S}^q $.
	Assume that $ M,N\in \mathbf{\Pi}_{q,r} $, $ N|N_{22} = 0 $ and $ \ker\Theta\subseteq\ker M|M_{22} $.
	Then, $ \mathcal{Z}_r^0(N)\subseteq \mathcal{Z}_r(M) $ if and only if 
	$
	\exists \alpha\geq0, \text{ s.t. } M-\alpha N\geq 0.
	$
\end{lemma}
%Lemma \ref{MSL} gives the necessary and sufficient conditions under which the set of solutions of an QMI are contained in that of another QMI.

\section{Data Informativity for MRC in Noiseless Case}

Consider a discrete-time linear input/state system:
\begin{equation}\label{n1}
x(t+1)=A_{\rm s}x(t)+B_{\rm s}u(t),
\end{equation}
and a discrete-time reference system model:
\begin{equation}\label{r}
x_{\rm m}(t+1)=A_{\rm m}x_{\rm m}(t)+B_{\rm m}r(t),
\end{equation}
where $ x(t) \in \mathbb{R}^n $ is the system state, $ u(t) \in \mathbb{R}^m $ is the control input, $ x_{\rm m}(t) \in \mathbb{R}^n $ is the reference state, and $ r(t)\in \mathbb{R}^p $ is the reference input with $ m \geq p $.
The system matrices $ A_{\rm s} \in \mathbb{R}^{n\times n}, B_{\rm s} \in \mathbb{R}^{n\times m} $ are unknown, whereas the reference model matrices $ A_{\rm m} \in \mathbb{R}^{n\times n} $ and $B_{\rm m} \in \mathbb{R}^{n\times p} $ are given, with $ A_{\rm m} $ being Schur.
The problem of MRC is to find a controller
\begin{equation}\label{controller}
u(t) = Kx(t) + Lr(t),
\end{equation} 
where the gains $ K\in \mathbb{R}^{m\times n} $ and $ L\in \mathbb{R}^{m\times p} $ are such that 
\begin{equation}\label{mc}
A_{\rm s}+B_{\rm s}K=A_{\rm m},\ B_{\rm s}L=B_{\rm m}.
\end{equation}
The rationale for \eqref{mc}, commonly referred to as matching conditions in the literature \cite{ioannou2006adaptive}, is the following: 
substituting \eqref{controller} into \eqref{n1} and making use of \eqref{mc}, the dynamics of the closed-loop system \emph{matches} the dynamics of the reference model \eqref{r}. 
In fact, the error $e(t)=x(t)-x_{\rm m}(t)$ has dynamics
$
e(t+1)= A_{\rm m} e(t),
$
implying convergence of the error to zero for any initial condition. 
%In other words, the dynamics of the closed-loop system matches the dynamics of the reference model \eqref{r}, hence the name matching conditions. 
Unfortunately, as $ A_{\rm s}, B_{\rm s} $ are unknown, a solution to \eqref{mc} cannot be computed. 
In the following, we will view MRC from the perspective of data informativity.
To this end, for given gains $ K $ and $ L $, let us define the set of systems that match the reference model \eqref{r} as
\begin{equation*}
\Sigma_{\rm exact}^{K,L}=\{(A,B)\mid A+BK=A_{\rm m}, BL=B_{\rm m}\}.
\end{equation*}
Obviously, if $ (A_{\rm s},B_{\rm s})\in\Sigma_{\rm exact}^{K,L} $, matching is possible via the control law \eqref{controller}. 

Consider now the case that only a set of input/state data of system \eqref{n1} is available:
\begin{equation}\label{d}
\begin{aligned}
X:=&\begin{bmatrix}
x(0)&x(1)&\cdots&x(T)
\end{bmatrix},\\
U_-:=&\begin{bmatrix}
u(0)&u(1)&\cdots&u(T-1)
\end{bmatrix}.
\end{aligned}
\end{equation}
At this point we stress that we do not make any a priori assumptions on the input data $ U_{-} $.
They could have been generated randomly, or from an open-loop or closed-loop experiment.
We aim to find conditions on the set of data $ (U_-,X) $ under which a controller \eqref{controller} can be designed so that the resulting closed-loop system matches the reference model \eqref{r}. 
To define this in a mathematical way, denote
\begin{equation}\label{sd}
X_+:=\begin{bmatrix}
x(1)&\cdots&x(T)
\end{bmatrix},
X_-:=\begin{bmatrix}
x(0)&\cdots&x(T-1)
\end{bmatrix},
\end{equation}
%The problem is to find necessary and sufficient conditions under which the set of data $ (U_-,X) $ allow to get gains $ K $ and $ L $ such that the matching conditions \eqref{mc} hold. 
and consider the set of all systems compatible with the data
\begin{align*}
\Sigma_{(U_-,X)}=\{(A,B)\mid X_+ = AX_- + BU_-\}.
\end{align*}
Since the data are generated by system \eqref{n1}, we must have $ (A_{\rm s},B_{\rm s})\in\Sigma_{(U_-,X)} $.
Yet, it is possible that $ \Sigma_{(U_-,X)} $ contains an infinite number of systems, and the data $ (U_-,X) $ do not allow to distinguish $ (A_{\rm s},B_{\rm s}) $ from any other system in $ \Sigma_{(U_-,X)} $.
Hence, we need to find gains $ K $ and $ L $ such that \emph{all} systems in $ \Sigma_{(U_-,X)} $ satisfy the matching condition with $ K $ and $ L $, leading to the definition of informativity for MRC.

\begin{definition}\label{d1} (Informativity for MRC).
	Let $ A_{\rm m} \in \mathbb{R}^{n\times n} $ be Schur and let $ B_{\rm m} \in \mathbb{R}^{n\times m} $. 
	The data $ (U_-,X) $, generated by system (\ref{n1}), are called informative for MRC if there exist control gains $ K $ and $ L $ such that $ \Sigma_{(U_-,X)} \subseteq \Sigma_{\rm exact}^{K,L} $.
\end{definition}
%The inclusion $ \Sigma_{(U_-,X)} \subseteq \Sigma_{{\rm exact}}^{K,L} $ claims that with the gains $ K $ and $ L $, all the systems explaining the data satisfies the matching condition.

The problem is now to find necessary and sufficient conditions for the data $ (U_-,X) $ to be informative for MRC, and, if so, finding suitable gains $ K $ and $ L $.
This is addressed by the following theorem.
%under which there exist $ K $ and $ L $ such that $ X_+ = AX_- + BU_- $ implies \eqref{mc}, which is solved by the following theorem.

\begin{theorem}\label{TDIMRC}
	The data $ (U_-,X) $, generated by system (\ref{n1}), are informative for MRC if and only if there exist $ V_1\in \mathbb{R}^{T\times n} $ and $ V_2\in \mathbb{R}^{T\times p} $ such that
	\begin{align}\label{T1}
	X_-V_1=I,\ X_+V_1=A_{\rm m},\ X_-V_2=0,\ X_+V_2=B_{\rm m}.
	\end{align}
	Moreover, if $ V_1 $ and $ V_2 $ satisfy (\ref{T1}), then the control gains $ K=U_-V_1,\ L=U_-V_2 $ are such that $ \Sigma_{(U_-,X)} \subseteq \Sigma_{{\rm exact}}^{K,L} $.
\end{theorem}

\begin{proof}
	(\underline{Sufficiency}):
	Define $ u=Kx+Lr $, where $ K=U_-V_1, L=U_-V_2 $.
	Then, for any $ (A,B) \in \Sigma_{(U_-,X)} $ we have
	\begin{align*}
	x(t\!+\!1)\!&=\!Ax(t)\!+\!BKx(t)\!+\!BLr(t)\\
	&=\!AX_-\!V_1x(t)\!+\!BU_-\!V_1x(t)\!+\!BU_-\!V_2r(t)\!+\!AX_-\!V_2r(t)\\
	&=\!X_+\!V_1x(t)\!+\!X_+\!V_2r(t)\!=\!A_{\rm m}x(t)\!+\!B_{\rm m}r(t).
	\end{align*}
	We obtain that for any $ (A,B)\in\Sigma_{(U_-,X)} $, the controller $ u=Kx+Lr $ with $ K=U_-V_1, L=U_-V_2 $ is such that $ \Sigma_{(U_-,X)}\subseteq \Sigma_{{\rm exact}}^{K,L} $, that is, $ (U_-,X) $ are informative for MRC.
	
	(\underline{Necessity}):
	Suppose there exist control gains $ K $ and $ L $ such that $ \Sigma_{(U_-,X)} \subseteq \Sigma_{{\rm exact}}^{K,L} $.
	Note that $ \Sigma_{(U_-,X)} $ is the solution set of the linear equation $ X_+ = AX_- + BU_- $. We denote the solution space of the corresponding homogeneous equation by:
	\vspace{-0.5ex}
	\begin{equation*}
	\Sigma^0_{(U_-,X)}=\{(A^0,B^0)\mid0= A^0X_- + B^0U_-\}.
	\end{equation*}
	Let $ (A^*,B^*)\in \Sigma_{(U_-,X)}, (A^0,B^0)\in \Sigma_{(U_-,X)}^0 $.
	Then, for any $\alpha\in \mathbb{R}$, 
	$ (A^*+\alpha A^0)X_-+(B^*+\alpha B^0)U_-=X_+ $,
	that is, $ (A^*+\alpha A^0,B^*+\alpha B^0)\in \Sigma_{(U_-,X)} $.
	Since $ \Sigma_{(U_-,X)} \subseteq \Sigma_{{\rm exact}}^{K,L} $, $ A^*+\alpha A^0+(B^*+\alpha B^0)K=A_{\rm m} $ and $ (B^*+\alpha B^0)L=B_{\rm m} $, which implies that $ A^0+B^0K=0 $, $ B^0L=0 $.
	Hence, \begin{equation*}
	\ker\begin{bmatrix}
	X_-^{\top}&U_-^{\top}
	\end{bmatrix}
	\subseteq 
	\ker\begin{bmatrix}
	I&K^{\top}
	\end{bmatrix},\ 
	\ker\begin{bmatrix}
	X_-^{\top}&U_-^{\top}
	\end{bmatrix}
	\subseteq \ker\begin{bmatrix}
	0&L^{\top}
	\end{bmatrix},
	\end{equation*} which is equivalent to 
	\begin{equation*}
	\text{im} \begin{bmatrix}
	I\\K
	\end{bmatrix} \subseteq \text{im} \begin{bmatrix}
	X_-\\U_-
	\end{bmatrix},\ \text{im} \begin{bmatrix}
	0\\L
	\end{bmatrix} \subseteq \text{im} \begin{bmatrix}
	X_-\\U_-
	\end{bmatrix}.
	\end{equation*}
	Then, there exists $ V_1 $ such that $ \begin{bmatrix}
	I\\K
	\end{bmatrix} = \begin{bmatrix}
	X_-\\U_-
	\end{bmatrix}V_1 $, implying
	\begin{equation*}
	X_-V_1=I,\
	X_+V_1=\begin{bmatrix}
	A^*&B^*
	\end{bmatrix}
	\begin{bmatrix}
	X_-\\U_-
	\end{bmatrix}V_1=A^*+B^*K=A_{\rm m}
	\end{equation*}
	and there exists  $ V_2 $ such that $ \begin{bmatrix}
	0\\L
	\end{bmatrix} = \begin{bmatrix}
	X_-\\U_-
	\end{bmatrix}V_2 $, implying
	\begin{equation*}\label{key}
	X_-V_2=0,\
	X_+V_2=\begin{bmatrix}
	A^*&B^*
	\end{bmatrix}
	\begin{bmatrix}
	X_-\\U_-
	\end{bmatrix}V_2=B^*L=B_{\rm m}.
	\end{equation*}
	We conclude that there exist $ V_1 $ and $ V_2 $ satisfying (\ref{T1}). 
\end{proof}
Condition \eqref{T1} is the same as \cite[eq. (18)]{breschi2021direct}.
However, \cite{breschi2021direct} assumes persistency of excitation of the input data, which is not required in Theorem \ref{TDIMRC}. In fact, Theorem \ref{TDIMRC} gives necessary and sufficient conditions under which a model reference controller can be found using the data.
These conditions are weaker than those required for unique system identification.
This is illustrated by the following example.

\begin{example}
	Consider the data set
	\begin{equation}\label{e1data}
	X=\begin{bmatrix}
	1&0&0&\frac{1}{2}\\
	1&0&1&0
	\end{bmatrix}\!,\
	U_-=\begin{bmatrix}
	1&-1&0\\
	1&-1&1
	\end{bmatrix}\!,
	\end{equation}
	and the model \eqref{r} with $ A_{\rm m}=\begin{bmatrix}
	-\frac{1}{2}&\frac{1}{2}\\0&-\frac{1}{2}
	\end{bmatrix} $ and 
	$ B_{\rm m}=\begin{bmatrix}
	0&0\\0&1
	\end{bmatrix} $.
	Because rank$ \begin{bmatrix}
	X_-\\
	U_-\\
	\end{bmatrix}=3 $, we know from \cite[Prop. 6]{van2020data} that the system cannot be uniquely identified.
	Indeed, given \eqref{e1data},
	\begin{equation*}
	\Sigma_{(U_-,X)} \! = \! \left\{\!\left(\!\begin{bmatrix}
	a-\frac{1}{2}&-a+\frac{1}{2}\\
	b&-b+1
	\end{bmatrix}\!,\begin{bmatrix}
	-a&a\\
	-b&b-1
	\end{bmatrix}\!\right)\! \bigg| a,b \in \mathbb{R}\right\},
	\end{equation*}
	that is, there is an infinite number of systems compatible with the data. 
	Despite this, by taking
	\begin{equation*}
	V_1\!=\!\begin{bmatrix}
	1&0\\0&-\frac{1}{2}\\-1&1
	\end{bmatrix}\!,\
	V_2\!=\!\begin{bmatrix}
	0&0\\0&1\\0&0
	\end{bmatrix}\!,
	\end{equation*}
	we obtain $ X_-V_1\!=\!I,\ X_+V_1\!=\!A_{\rm m}, X_-V_2\!=\!0, X_+V_2\!=\!B_{\rm m} $.
	Then, the gains solving the matching conditions for all systems in $ \Sigma_{(U_-,X)} $ are $ K\!=\!U_-V_1\!=\!\begin{bmatrix}
	1&\frac{1}{2}\\0&\frac{3}{2}
	\end{bmatrix}\! $ and $ L\!=\!U_-V_2\!=\!\begin{bmatrix}
	0&-1\\0&-1
	\end{bmatrix}\! $.\\
	%In other words, Theorem \ref{TDIMRC} gives the necessary and sufficient conditions for which the MRC problem can be solved even when the system \eqref{n1} cannot be uniquely identified.
	%Moreover, it can also be verified that with $ K,L $ given above and $ \alpha_1=3, \alpha_2=1 $, the LMIs \eqref{T1K} and \eqref{T1L} hold.
\end{example}

To establish a bridge between the noiseless setting and the noisy setting that will be studied in Section IV, it is convenient to present necessary and sufficient conditions for data informativity in terms of linear matrix inequalities (LMIs).
%This endeavor serves to establish a connection between the outcomes obtained in noiseless scenario and those in noisy scenario.

\begin{corollary}\label{CDI}
	The data $ (U_-,X) $, generated by system (\ref{n1}), are informative for MRC if and only if there exist gains $ K,L $ and $ \alpha_1,\alpha_2>0 $ such that
	\begin{align}\label{T1K}
	\renewcommand{\arraystretch}{1.12}
	\begin{bmatrix}
	X_+X_+^{\top}&-X_+X_-^{\top}&-X_+U_-^{\top}&-A_{\rm m}\\
	-X_-X_+^{\top}&X_-X_-^{\top}&X_-U_-^{\top}&I\\
	-U_-X_+^{\top}&U_-X_-^{\top}&U_-U_-^{\top}&K\\
	-A_{\rm m}^{\top}&I&K^{\top}&\alpha_1I
	\end{bmatrix}\geq 0,\\
	\label{T1L}
	\renewcommand{\arraystretch}{1.12}
	\begin{bmatrix}
	X_+X_+^{\top}&-X_+X_-^{\top}&-X_+U_-^{\top}&-B_{\rm m}\\
	-X_-X_+^{\top}&X_-X_-^{\top}&X_-U_-^{\top}&0\\
	-U_-X_+^{\top}&U_-X_-^{\top}&U_-U_-^{\top}&L\\
	-B_{\rm m}^{\top}&0&L^{\top}&\alpha_2I
	\end{bmatrix}\geq 0.
	\end{align}
\end{corollary}

\begin{proof}
	(\underline{Sufficiency}):
	Suppose \eqref{T1K} and \eqref{T1L} hold.
	By computing the Schur complement of (\ref{T1K}) with respect to its fourth diagonal block, we obtain
	\begin{equation}\label{TP1K}
	\alpha_1\begin{bmatrix}
	X_+\\
	-X_-\\
	-U_-\\
	\end{bmatrix}
	\begin{bmatrix}
	X_+\\
	-X_-\\
	-U_-\\
	\end{bmatrix}^{\top}-
	\begin{bmatrix}
	-A_{\rm m}\\
	I\\
	K\\
	\end{bmatrix}\begin{bmatrix}
	-A_{\rm m}\\
	I\\
	K\\
	\end{bmatrix}^{\top}\geq 0.
	\end{equation}
	For any $ x \in \ker \begin{bmatrix}
	X_+^{\top}&
	-X_-^{\top}&
	-U_-^{\top}
	\end{bmatrix} $, we obtain from \eqref{TP1K} that 
	\begin{equation*}\label{key}
	x^{\top}\begin{bmatrix}
	-A_{\rm m}\\
	I\\
	K\\
	\end{bmatrix}\begin{bmatrix}
	-A_{\rm m}\\
	I\\
	K\\
	\end{bmatrix}^{\top}x\leq 0 \Rightarrow x \in \ker \begin{bmatrix}
	-A_{\rm m}\\
	I\\
	K\\
	\end{bmatrix}^{\top}.
	\end{equation*}
	Hence, we have 
	\begin{equation*}\label{key}
	\ker \begin{bmatrix}
	X_+\\
	-X_-\\
	-U_-\\
	\end{bmatrix}^{\top} \!\! \subseteq
	\ker \begin{bmatrix}
	-A_{\rm m}\\
	I\\
	K\\
	\end{bmatrix}^{\top} \! \!\Leftrightarrow \text{im} \begin{bmatrix}
	-A_{\rm m}\\
	I\\
	K\\
	\end{bmatrix} \! \subseteq 
	\text{im} \begin{bmatrix}
	X_+\\
	-X_-\\
	-U_-\\
	\end{bmatrix}\!\!.
	\end{equation*}
	Then, there exists $ V_1\in \mathbb{R}^{T\times n} $ such that
	\begin{equation}\label{V}
	\begin{bmatrix}
	-A_m \\	I \\ K
	\end{bmatrix} = \begin{bmatrix}
	X_+\\	-X_-\\	-U_-
	\end{bmatrix} V_1.
	\end{equation}
	Along similar steps, we obtain from \eqref{T1L} that there exists $ V_2\in \mathbb{R}^{T\times p} $ such that 
	%\vspace{-0.3em}
	\begin{equation}\label{W}
	\begin{bmatrix}
	-B_m \\	0 \\ L
	\end{bmatrix}  = 
	\begin{bmatrix}
	X_+\\	-X_-\\	-U_-^\top 
	\end{bmatrix} V_2.
	\end{equation}
	Thus, according to Theorem \ref{TDIMRC}, the data $ (U_-,X) $ are informative for MRC.
	
	(\underline{Necessity}):
	Suppose there exist gains $ K $ and $ L $ such that $ \Sigma_{(U_-,X)} \subseteq \Sigma_{{\rm exact}}^{K,L} $.
	According to Theorem \ref{TDIMRC}, there exist $ V_1 $ and $ V_2 $ such that \eqref{V} and \eqref{W} hold.
	Then, there must exist a large enough $ \alpha_1>0 $ such that
	\begin{align*}
	\alpha_1&\begin{bmatrix}
	X_+\\
	-X_-\\
	-U_-\\
	\end{bmatrix}
	\begin{bmatrix}
	X_+\\
	-X_-\\
	-U_-\\
	\end{bmatrix}^{\top}-
	\begin{bmatrix}
	-A_{\rm m}\\
	I\\
	K\\
	\end{bmatrix}
	\begin{bmatrix}
	-A_{\rm m}\\
	I\\
	K\\
	\end{bmatrix}^{\top}\\
	=&
	\begin{bmatrix}
	X_+\\
	-X_-\\
	-U_-\\
	\end{bmatrix}(\alpha_1I-V_1 V_1 \!^{\top})
	\begin{bmatrix}
	X_+\\
	-X_-\\
	-U_-\\
	\end{bmatrix}^{\top}\geq0.
	\end{align*}
	Since $ \begin{bmatrix}
	X_+\\
	-X_-\\
	-U_-\\
	\end{bmatrix}
	\begin{bmatrix}
	X_+\\
	-X_-\\
	-U_-\\
	\end{bmatrix}^{\top}-\begin{bmatrix}
	A_{\rm m}\\
	I\\
	K\\
	\end{bmatrix}(\alpha_1I)^{-1}
	\begin{bmatrix}
	A_{\rm m}\\
	I\\
	K\\
	\end{bmatrix}^{\top} \geq 0 $ and $ \alpha_1I > 0 $, we obtain the LMI \eqref{T1K} using a Schur complement argument.
	Along similar steps, we also conclude existence of a large enough $ \alpha_2>0 $ such that \eqref{T1L} holds.
\end{proof}

%Later we will see that Corollary \ref{CDI} develops an interesting connection with Theorem \ref{TDIAMRC} (noisy case).
As compared to Theorem \ref{TDIMRC}, the significance in Corollary \ref{CDI} is mostly theoretical. 
As we will see in Theorem \ref{TDIAMRCSG}, the LMI conditions for noisy data boil down to those in Corollary \ref{CDI} in the special case of noiseless data.

\section{Approximate MRC Using Noisy Data}
Consider a discrete-time linear input/state system:
\begin{equation}\label{1}
x(t+1)=A_{\rm s}x(t)+B_{\rm s}u(t)+w(t),
\end{equation}
where $ x(t),u(t),A_{\rm s} $ and $ B_{\rm s} $ are as before, and $ w(t)\in \mathbb{R}^n $ is an unknown noise term.
The presence of the unknown noise makes it impossible, in general, to achieve exact data-driven matching.
In fact, we later show in Theorem \ref{TDIAMRC} that the matching conditions can hardly be satisfied with noisy data.
This motivates relaxing the exact matching conditions \eqref{mc} to approximate matching conditions.

Let $ D^A\geq0 $, $ D^B\geq0 $ be difference matrices, and $ \Gamma^A=\text{diag}(\gamma^{A}_1,\gamma^{A}_2,\dots,\gamma^{A}_n)>0 $, $ \Gamma^B=\text{diag}(\gamma^{B}_1,\gamma^{B}_2,\dots,\gamma^{B}_p) >0 $ be weight matrices, which determine the ``distance" between the closed-loop system and the reference model by
%Consider the column decomposition $ A_{\rm s}+B_{\rm s}K-A_{\rm m}=[a_1\ a_2\ \cdots\ a_n] $, $ B_{\rm s}L-B_{\rm m}=[b_1\ b_2\ \cdots\ b_p] $.
%We specify the conditions for approximate model matching as follows. 
%For given $ \gamma_i^A $ and $ D^A $, and $ \gamma_i^B $ and $ D^B $, we require that:
%$
%\sum_{i=1}^{n} \gamma^{A}_ia_ia_i^\top \leq D^A,\ \sum_{i=1}^{p} \gamma^{B}_ib_ib_i^\top \leq D^B.
%$
%This is equivalent to
\begin{align}\label{a}
&D^A-(A_{\rm s}+B_{\rm s}K-A_{\rm m})\Gamma^A(A_{\rm s}+B_{\rm s}K-A_{\rm m})^\top \geq 0,\! \\
\label{b}
&D^B-(B_{\rm s}L-B_{\rm m})\Gamma^B(B_{\rm s}L-B_{\rm m})^\top \geq 0.
\end{align}
%where the difference matrices $ D^A\geq0 $, $ D^B\geq0 $ and the weight matrices $ \Gamma^A=\text{diag}(\gamma^{A}_1,\gamma^{A}_2,\dots,\gamma^{A}_n)>0 $, $ \Gamma^B=\text{diag}(\gamma^{B}_1,\gamma^{B}_2,\dots,\gamma^{B}_p) $ $ >0 $ determine "distance" allowed from the ideal matching.
In other words, \eqref{a}-\eqref{b} determine approximate matching by allowing $ A_{\rm s} + B_{\rm s}K - A_{\rm m} $ and $ B_{\rm s}L - B_{\rm m} $ to be ``small".
%, where the matrices $ D^A $, $ D^B $, $ \Gamma^A $ and $ \Gamma^B $ can be designed to give different importance to the components of $ e $. 
%Specifically, $ D^A $, $ D^B $ determine the structure of $ e $.
%For instance, if there is a higher precision requirement on the first component $ e_1 $ of $ e $, then we can assign small values to the top-left elements $ d_{11}^A $ and $ d_{11}^B $ of $ D^A $, $ D^B $.
%In this way, the first components of all $ a_1,\dots,a_n,b_1,\dots,b_{\rm m} $ will be small, which will also lead to a small $ e_1 $ according to \eqref{e}.
%Moreover, $ \Gamma^A $, $ \Gamma^B $ determine the effect of each component of $ x $ and $ r $ on $ e $.
%For instance, if the first component of $ x $ is large in most time, then we can assign a large value to $ \gamma_{1}^A $.
%In this way, the modulus of $ a_1 $ will be small, which will result in a ``smaller'' $ e $ according to \eqref{e}.
For given $ K,L $, define the set of all systems that satisfy the conditions for approximate matching as
\begin{align*}
\Sigma^{K,L}\!=\! \{(A,B)|&D^A\!\geq\!(A\!+\!BK\!-\!A_{\rm m})\Gamma^A(A\!+\!BK\!-\!A_{\rm m})\!^\top\\ \text{and } &D^B\!\geq\!(BL\!-\!B_{\rm m})\Gamma^B(BL\!-\!B_{\rm m})\!^\top\}.
\end{align*}

Let $ (U_-,X) $ be a set of input/state data generated by the noisy system \eqref{1}, and let
\begin{equation*}
W_-:=\begin{bmatrix}
w(0)&\cdots&w(T-1)
\end{bmatrix}
\end{equation*}
be the matrix of unknown noise samples.
Assume the unknown noise $ W_- $ satisfies the noise model
\begin{equation}\label{noise}
\begin{bmatrix}
I \\ W_-^\top
\end{bmatrix}^\top
\Phi
\begin{bmatrix}
I \\ W_-^\top
\end{bmatrix} \geq 0
\end{equation}
for some $ \Phi=\begin{bmatrix}
\Phi_{11}&\Phi_{12}\\
\Phi_{21}&\Phi_{22}\\
\end{bmatrix} \in \mathbf{\Pi}_{n,T} $.
Let us recall that the noise model \eqref{noise} can capture energy bounds, individual sample bounds and sample covariance bounds even within a subspace \cite{van2023quadratic}.
%We aim to find conditions for the set of noisy data $ (U_-,X) $ under which a controller can be designed to make the closed-loop system close to \eqref{r}, in a sense defined more precisely below.
We define the set of all systems explaining the data $ (U_-, X) $, i.e., all $ (A, B) $ satisfying
\begin{equation}\label{edata}
X_+ = AX_- + BU_- + W_-
\end{equation}
for some $ W_- $ satisfying \eqref{noise}, denoted by 
\begin{equation}\label{sigma}
\Sigma_{(U_-,X)}=\{(A,B) \mid \eqref{edata} \text{ holds for some } W_-^\top\in\mathcal{Z}_{T}(\Phi)\}.
\end{equation}
Similar to Section III, we must have $ (A_{\rm s},B_{\rm s})\in\Sigma_{(U_-,X)} $, and we need to find gains $ K $ and $ L $ such that \textit{all} systems in $ \Sigma_{(U_-,X)} $ approximately match \eqref{r}, which leads to the definition of informativity for approximate MRC.

\begin{definition}\label{D} (Informativity for approximate MRC).
	Let $ A_{\rm m} $, $ B_{\rm m} $, $ D^A $, $ D^B $, $ \Gamma^A $, $ \Gamma^B $ be the matrices in \eqref{r}, \eqref{a}-\eqref{b}.
	The data $ (U_-,X) $, generated by system (\ref{1}) with noise model (\ref{noise}), are called informative for approximate MRC if there exist gains $ K $ and $ L $ such that $ \Sigma_{(U_-,X)} \subseteq \Sigma^{K,L} $.
\end{definition}

We now aim to find conditions for the set of noisy data $ (U_-,X) $ to be informative for approximate MRC.
We tackle this problem in the framework of QMIs.
Denote
\begin{align}
\label{MK}
M^K&=\begin{bmatrix}
D^A-A_{\rm m}\Gamma^AA_{\rm m}^{\top}&A_{\rm m}\Gamma^A&A_{\rm m}\Gamma^AK^{\top},\\
\Gamma^AA_{\rm m}^{\top}&-\Gamma^A&-\Gamma^AK^{\top}\\
K\Gamma^AA_{\rm m}^{\top}&-K\Gamma^A&-K\Gamma^AK^{\top}\\
\end{bmatrix},\\
\label{ML}
M^L&=\begin{bmatrix}
D^B-B_{\rm m}\Gamma^BB_{\rm m}^{\top}&0&B_{\rm m}\Gamma^BL^{\top}\\
0&0&0\\
L\Gamma^BB_{\rm m}^{\top}&0&-L\Gamma^BL^{\top}\\
\end{bmatrix},\\
\label{N}
N&=\begin{bmatrix}
I&X_+\\
0&-X_-\\
0&-U_-\\
\end{bmatrix}
\begin{bmatrix}
\Phi_{11}&\Phi_{12}\\
\Phi_{21}&\Phi_{22}\\
\end{bmatrix}
\begin{bmatrix}
I&X_+\\
0&-X_-\\
0&-U_-\\
\end{bmatrix}^\top.
\end{align}
Then, $ D^A\geq(A+BK-A_m)\Gamma^A(A+BK-A_m)^\top $ is equivalent to
%\vspace{-0.2em}
\begin{equation}\label{PMK}
\begin{bmatrix}
I\\A^\top\\B^\top
\end{bmatrix}^\top
M^K
\begin{bmatrix}
I\\A^\top\\B^\top
\end{bmatrix}^\top \geq 0,
\end{equation}
and $ D^B\geq(BL-B_m)\Gamma^B(BL-B_m)^\top $ is equivalent to
%\vspace{-0.1em}
\begin{equation}\label{PML}
\begin{bmatrix}
I\\A^\top\\B^\top
\end{bmatrix}^\top
M^L
\begin{bmatrix}
I\\A^\top\\B^\top
\end{bmatrix}^\top \geq 0.
\end{equation}
Thus, $ \Sigma^{K,L} $ is equivalently written as
\begin{equation*}
\Sigma^{K,L}=\{(A,B) \mid \eqref{PMK} \text{ and } \eqref{PML} \text{ hold} \}.
\end{equation*}
Meanwhile, \eqref{noise} and \eqref{sigma} indicate that $ (A,B) \in \Sigma_{(U_-,X)} $ if and only if $ (A,B) $ satisfies the QMI
%\vspace{-0.1em}
\begin{equation}\label{data}
\begin{bmatrix}
I\\A^\top\\B^\top
\end{bmatrix}^\top
N
\begin{bmatrix}
I\\A^\top\\B^\top
\end{bmatrix}^\top\geq 0.
\end{equation}
Hence, our goal in the noisy case is to find necessary and sufficient conditions on the data $ (U_-,X) $ under which there exist control gains $ K $ and $ L $ such that the QMI \eqref{data} implies the QMIs \eqref{PMK} and \eqref{PML}.
%Using the matrix S-lemma (Lemma \ref{MSL}) and matrix Finsler's lemma (Lemma \ref{MFL}), we now provide a necessary and sufficient condition for this type of implication. 
This is resolved in the following theorem. 

%Finally, we apply Theorem \ref{TI} and Lemma \ref{MSL} to obtain the necessary and sufficient condition for the informativity and develop the following theorem.

\begin{theorem}\label{TDIAMRC}
	Let $ \Phi \in \mathbf{\Pi}_{n,T} $. 
	Assume that	$ (i)\ N \not\leq 0 $; or $ (ii)\ D^A=D^B=0 $.
	Then, the data $ (U_-,X) $, generated by system (\ref{1}) with noise model (\ref{noise}), are informative for approximate MRC if and only if there exist $ K,L $ and $ \alpha_1,\alpha_2>0 $ such that
	\begin{equation}\label{TK}
	\begin{bmatrix}
	D^A&0&0&-A_{\rm m}\\
	0&0&0&I\\
	0&0&0&K\\
	-A_{\rm m}^{\top}&I&K^{\top}&(\Gamma^A)^{-1}
	\end{bmatrix}-
	\alpha_1\begin{bmatrix}
	I&X_+\\
	0&-X_-\\
	0&-U_-\\
	0&0
	\end{bmatrix}\Phi
	\begin{bmatrix}
	I&X_+\\
	0&-X_-\\
	0&-U_-\\
	0&0
	\end{bmatrix}^{\!\top}\!\geq 0,
	\end{equation}
	\begin{equation}\label{TL}
	\begin{bmatrix}
	D^B&0&0&-B_{\rm m}\\
	0&0&0&0\\
	0&0&0&L\\
	-B_{\rm m}^{\top}&0&L^{\top}&(\Gamma^B)^{-1}
	\end{bmatrix}-
	\alpha_2\begin{bmatrix}
	I&X_+\\
	0&-X_-\\
	0&-U_-\\
	0&0
	\end{bmatrix}\Phi
	\begin{bmatrix}
	I&X_+\\
	0&-X_-\\
	0&-U_-\\
	0&0
	\end{bmatrix}^{\!\top}\!\geq 0.
	\end{equation}
\end{theorem}

\begin{proof}
The proof makes use of the matrix S-lemma (Lemma \ref{MSL}) and matrix Finsler's lemma (Lemma \ref{MFL}).
	
	(\underline{Sufficiency}):
	Suppose there exist $ K,L $ and $ \alpha_1,\alpha_2>0 $ such that \eqref{TK} and \eqref{TL} holds.
	By calculating the Schur complement of \eqref{TK} and \eqref{TL} with respect to their fourth diagonal block, we obtain that
	\begin{equation}\label{PK}
	M^K-\alpha_1N\geq 0,
	\end{equation}
	\begin{equation}
	 M^L-\alpha_2N\geq 0.
	\end{equation}
	Since $ \forall (A,B) \in \Sigma_{(U_-,X)} $, (\ref{data}) holds, we have that \eqref{PMK}-\eqref{PML} hold, which implies $ (A,B) \in \Sigma^{K,L} $.
	Hence, $ \Sigma_{(U_-,X)} \subseteq \Sigma^{K,L} $, i.e., the data $ (U_-,X) $ are informative for approximate MRC.
	
	(\underline{Necessity}):
	Suppose $ (U_-,X) $ are informative for approximate MRC, that is, there exist gains $ K, L $, such that $ \Sigma_{(U_-,X)} \! \subseteq \!  \Sigma^{K,L} $.
	Since the data $ (U_-,X) $ are generated by \eqref{1}, we have that $ (A_{\rm s},B_{\rm s})  \! \in \!  \Sigma_{(U_-,X)} $.
	Since $  \Sigma_{(U_-,X)}  \! \subseteq \!  \Sigma^{K,L} $, any $ (A,B) \! \in \! \Sigma _{(U_-,X)}$ satisfies \eqref{PMK}-\eqref{PML}, which implies $ \mathcal{Z}_{n+m}(N) \! \subseteq \!  \mathcal{Z}_{n+m}(M^K) $ and $ \mathcal{Z}_{n+m}(N) \! \subseteq \!  \mathcal{Z}_{n+m}(M^L) $.\\
	 To apply Lemma \ref{MSL} and Lemma \ref{MFL}, we first show that $ N \in \mathbf{\Pi}_{n,n+m} $.
	With the assumption $ \Phi \in \mathbf{\Pi}_{n,T} $, we have $ \Phi_{22}\leq0 $, which implies that
	\begin{equation*}
	N_{22}=\begin{bmatrix}
	-X_-\\
	-U_-\\
	\end{bmatrix}\Phi_{22}
	\begin{bmatrix}
	-X_-\\
	-U_-\\
	\end{bmatrix}^{\top}\leq 0.
	\end{equation*}
	Then, since $ \ker\Phi_{22}\subseteq \ker\Phi_{12} $, for any $ x \in \ker N_{22}=
	\ker \Phi_{22}
	\begin{bmatrix}
	-X_-^{\top}&-U_-^{\top}
	\end{bmatrix}, $ we have
	\begin{equation*}
	N_{12}x=(\Phi_{12}+X_+\Phi_{22})
	\begin{bmatrix}
	-X_-\\
	-U_-\\
	\end{bmatrix}^{\top} x =0.
	\end{equation*}
	Hence, we have $ \ker N_{22} \subseteq \ker N_{12} $.
	Finally, according to \eqref{Pi}, $ (A_{\rm s},B_{\rm s}) \in \Sigma_{(U_-,X)} $ implies
	$
	N|N_{22}\geq 0,
	$
	that is, $ N \in \mathbf{\Pi}_{n,n+m} $.\\
	We now suppose that assumption $ (i) $ holds.
	Since $ N\not\leq0 $, we can apply Lemma \ref{MSL} with $ \mathcal{Z}_{n+m}(N)\subseteq \mathcal{Z}_{n+m}(M^K) $, obtaining that there exists $ \alpha_1\geq 0 $ such that (\ref{PK}) holds.
	Rewrite (\ref{PK}) as 
	$$ \begin{bmatrix}
	D^A&0&0\\
	0&0&0\\
	0&0&0\\
	\end{bmatrix}-
	\alpha_1N-
	\begin{bmatrix}
	-A_{\rm m}\\
	I\\
	K\\
	\end{bmatrix}\Gamma^A
	\begin{bmatrix}
	-A_{\rm m}\\
	I\\
	K\\
	\end{bmatrix}^\top\!\geq 0. $$
	Using a Schur complement argument, we obtain (\ref{TK}).
	Meanwhile, the $ (2, 2) $ block of (\ref{PK}) yields $ -\Gamma^A-\alpha_1X_-\Phi_{22}X_-^{\top}\geq 0 $, where $ \Gamma^A>0 $ implies $ \alpha_1\neq 0 $.
	Similarly, from $ \mathcal{Z}_{n+m}(N)\subseteq \mathcal{Z}_{n+m}(M^L) $, we have that there exists $ \alpha_2>0 $ such that (\ref{TL}) holds.
	\\
	Next, we suppose that assumption $ (ii) $ holds.
	Since the case $ N \not\leq 0 $ has been addressed, we only need to consider the case $ N \leq 0 $, which implies $ \mathcal{Z}_{n+m}(N)=\mathcal{Z}^0_{n+m}(N) $.\\
	We now verify the assumptions of Lemma \ref{MFL}. 
	First, using a Schur complement argument, $ N \leq 0 $ implies $ N|N_{22}\leq 0 $.
	On the other hand, $ N \in \mathbf{\Pi}_{n,n+m} $ implies $ N|N_{22}\geq 0 $.
	Hence, we have $ N|N_{22} = 0 $.
	Then, we show that $ M^K, M^L \in \mathbf{\Pi}_{n,n+m} $.
	Since the data $ (U_-,X) $ are generated by \eqref{1} and $ D^A=D^B=0 $, we have $ (A_{\rm s},B_{\rm s}) \in \Sigma_{(U_-,X)} \subseteq \Sigma^{K,L}_{\rm exact} = \{(A,B)|A+BK=A_{\rm m}, BL=B_{\rm m}\} $.
	According to \eqref{MK}, we have
	\begin{align*}
	M^K_{22}=&\begin{bmatrix}
	-\Gamma^A&-\Gamma^AK^{\top}\\
	-K\Gamma^A&-K\Gamma^AK^{\top}\\
	\end{bmatrix}=-\begin{bmatrix}
	I\\
	K
	\end{bmatrix}\Gamma^A
	\begin{bmatrix}
	I\\
	K
	\end{bmatrix}^{\top}\!\leq0,
	\\
	%	M^L_{22}=&\begin{bmatrix}
	%	0&0\\
	%	0&-L\Gamma^BL^{\top}\\
	%	\end{bmatrix}=-\begin{bmatrix}
	%	0\\
	%	L
	%	\end{bmatrix}\Gamma^B
	%	\begin{bmatrix}
	%	0\\
	%	L
	%	\end{bmatrix}^{\top}\leq0,
	%	\\
	M^K_{12}=&\begin{bmatrix}
	A_{\rm m}\Gamma^A&A_{\rm m}\Gamma^AK^{\top}\\
	\end{bmatrix}=\begin{bmatrix}
	A_{\rm s}&B_{\rm s}
	\end{bmatrix}
	\begin{bmatrix}
	I\\
	K
	\end{bmatrix}\Gamma^A
	\begin{bmatrix}
	I\\
	K
	\end{bmatrix}^{\top}\\
	=&\begin{bmatrix}
	A_{\rm s}&B_{\rm s}
	\end{bmatrix}M^K_{22},\\
	M|M^K_{22}=&M^K_{11}-M^K_{12}(M^K_{22})^{\dagger}M^K_{21}\\
	=&M^K_{11}-\begin{bmatrix}
	A_{\rm s}&B_{\rm s}
	\end{bmatrix}M^K_{22}(M^K_{22})^{\dagger}M^K_{22}\begin{bmatrix}
	A_{\rm s}&B_{\rm s}
	\end{bmatrix}^{\top}\\
	=& -A_{\rm m}\Gamma^AA_{\rm m}^{\top}+
	\begin{bmatrix}
	A_{\rm s}&B_{\rm s}
	\end{bmatrix}\begin{bmatrix}
	I\\
	K
	\end{bmatrix}\Gamma^A\begin{bmatrix}
	I\\
	K
	\end{bmatrix}^{\top}\!
	\begin{bmatrix}
	A_{\rm s}&B_{\rm s}
	\end{bmatrix}^{\top}\\
	=&-A_{\rm m}\Gamma^AA_{\rm m}^{\top}+A_{\rm m}\Gamma^AA_{\rm m}^{\top}=0.
	\end{align*}
	Obviously, $ \ker M^K_{22} \subseteq\ker M^K_{12} $.
	Similarly, we can obtain from \eqref{ML} that $ M^L_{22}\leq 0 $, $ \ker M^L_{22}\subseteq\ker M^L_{12} $ and $ M|M^L_{22}=0 $.
	Hence, $ M^K, M^L \in \mathbf{\Pi}_{n,n+m} $.
	Finally, for any $ \Theta \in \mathbb{S}^n $, we have that $ \ker\Theta \subseteq \ker M^K|M^K_{22} = \ker M^L|M^L_{22} =\mathbb{R}^n $.\\
	Having verified its assumptions, we apply Lemma \ref{MFL} with $ \mathcal{Z}_{n+m}(N)=\mathcal{Z}^0_{n+m}(N)\subseteq \mathcal{Z}_{n+m}(M^K) $ and $ \mathcal{Z}_{n+m}(N)=\mathcal{Z}^0_{n+m}(N)\subseteq \mathcal{Z}_{n+m}(M^L) $, obtaining that there exist $ \alpha_1,\alpha_2\geq 0 $ such that $	M^K-\alpha_1N\geq 0 $, $ M^L-\alpha_1N\geq 0 $.
	Using a Schur complement argument and zooming in on the $ (2, 2) $ block as before, we have that there exist $ \alpha_1, \alpha_2>0 $ such that (\ref{TK}) and (\ref{TL}) hold.
	\\
	We thus conclude that under assumption $ (i) $ or $ (ii) $, there exist $ K,L $ and $ \alpha_1,\alpha_2>0 $ such that (\ref{TK}) and (\ref{TL}) hold.
\end{proof}

%Let us make use of an example to show the necessity for considering approximate MRC.
%We first give a necessary condition for informativity for approximate MRC.
%
%\begin{lemma}\label{LIN}
%	Suppose that the data $ (U−, X) $ are informative	for approximate MRC, that is, there exists gains $ K $ and $ L $ such that $ \Sigma_{(U_-,X)} \subseteq \Sigma^{K,L} $.
%	Then, 
%	\begin{equation*}
%	\text{im}\begin{bmatrix}
%	I\\K
%	\end{bmatrix}\subseteq
%	\text{im}\begin{bmatrix}
%	X_-\\U_-
%	\end{bmatrix}.
%	\end{equation*}
%\end{lemma}
%\begin{proof}
%	content...
%\end{proof}
%
%\begin{lemma}\label{LE}
%	(Theorem 3.1 in ). 
%	content...
%\end{lemma}
%
As explained at the beginning of this section, approximate MRC, as studied in Theorem \ref{TDIAMRC}, is relevant in the noisy data setting because exact MRC is not possible in most cases.
We illustrate this by supposing that the data $ (U_-, X) $ are generated by \eqref{1} with $ \Phi_{22}=-I $ and $ \Phi_{12}=0 $ in \eqref{noise}, giving a widely considered bounded-energy noise with $ W_-W_-^{\top}\leq \Phi_{11} $, see, e.g., \cite[P2254, (i)]{van2023quadratic}.
Clearly, it is almost never the case that $ W_-W_-^{\top} = \Phi_{11} $, since this would imply that the unknown noise is precisely equal to the upper bound $ \Phi_{11} $.
%It is seen that $ W_-W_-^{\top} = \Phi_{11} $ can hardly hold since the noise are unknown.
Then, in most cases, we have $ W_-W_-^{\top} \neq\Phi_{11} $, i.e.,
\begin{equation*}
\Phi_{11}\!\!-\!\!W_-W_-^{\top}\!\!\!
=\!\!\begin{bmatrix}
I\\
A_{\rm s}^\top\\
B_{\rm s}^\top\\
\end{bmatrix}^{\!\!\top}\!\!
\underbrace{
\begin{bmatrix}
I&X_+\\
0&-X_-\\
0&-U_-\\
\end{bmatrix}\!\!\!
\begin{bmatrix}
\Phi_{11}&0\\
0&-I\\
\end{bmatrix}\!\!\!
\begin{bmatrix}
I&X_+\\
0&-X_-\\
0&-U_-\\
\end{bmatrix}^{\!\!\top}\!\!
}_{N}
\begin{bmatrix}
I\\
A_{\rm s}^\top\\
B_{\rm s}^\top\\
\end{bmatrix}
\end{equation*}
has at least one positive eigenvalue, which implies $ N \not\leq 0 $ as in $ (i) $.
%%We claim that in this case, the matching conditions \eqref{mc} will never hold, that is $ D^A\neq 0 $ or $ D^B\neq 0 $.
%%To show this, let $ \Phi_{11}>0 $ and $ D^A=D^B=0 $.
%In this case, we suppose $ (U_-, X) $ is a set of informative data for approximate MRC.
%Denote the dimension of the kernel of $ \begin{bmatrix}
%I&0&0\\
%X_+^{\top}&-X_-^{\top}&-U_-^{\top}\\
%\end{bmatrix} $ as $\delta$.
%According to Lemma \ref{LIN}, we have that
%\begin{align*}
%\delta&=\dim\ker\begin{bmatrix}
%I&0&0\\
%X_+^{\top}&-X_-^{\top}&-U_-^{\top}\\
%\end{bmatrix}\\
%&=\dim\ker\begin{bmatrix}
%X_-^{\top}&U_-^{\top}\\
%\end{bmatrix}\\
%&=n+m-\rank\begin{bmatrix}
%X_-\\U_-
%\end{bmatrix}<m.
%\end{align*}
%Denote the positive inertia coefficient of matrix $ N $ and $ \Phi $ by $ \pi_{N} $ and $ \pi_{\Phi} $, respectively.
%Since the dimension of state $ n $ is usually greater than or equal to the dimension of input $ m $, we can obtain from \eqref{N} and Lemma \ref{LE} that,
%which implies $ N\not\leq 0 $.
If we further impose exact matching with $ D^A=D^B=0 $, then it can be shown that \eqref{TK} and \eqref{TL} are never feasible, which shows the need for relaxing exact MRC.

\begin{remark}
	In practice, one may wish to optimize over the matrices  $ D^A $ and $ D^B $. 
	Because the conditions \eqref{TK} and \eqref{TL} are linear in $ D^A $ and $ D^B $, this can be conveniently done by formulating the semidefinite program that aims at minimizing tr$ (D^A) +\text{tr}(D^B) $, subject to \eqref{TK} and \eqref{TL}.
	This is illustrated in Section VI by means of a numerical example.
	% solving \eqref{TK} and \eqref{TL} for the minimum trace of $ D^A $ and $ D^B $, as illustrated in the numerical example in Section VI.
\end{remark}

We conclude this section with an example highlighting a drawback of approximate MRC, which provides a motivation for the next section.
\begin{example}\label{e2}
	Consider the data set 
	\begin{equation}\label{e2d}
	\begin{aligned}
	X=&\begin{bmatrix}
	0&1&0&-1&0&1&0&-1&0&1
	\end{bmatrix},\\
	U_-=&\begin{bmatrix}
	1&-1&-1&1&1&-1&-1&1&1
	\end{bmatrix},
	\end{aligned}
	\end{equation}
	with $ T=10 $ and noise satisfying \eqref{noise} with $\Phi=$ diag$ (0.1,-1,-1,-1,-1,-1,-1,-1,-1,-1) $.
	For given $ A_{\rm m}=0.9 $, $ D^A=0.2 $, $ \Gamma^A=1 $, we solve \eqref{TK} and obtain $ K=0.1, \alpha_1=1 $.
	For given $ B_{\rm m}=1 $, $ D^B=0.1 $, $ \Gamma^B=1 $, we solve \eqref{TL} and obtain $ L=1, \alpha_2=0.5 $.
	However, \eqref{e2d} implies that $ (1,1)\in \Sigma_{(U_-,X)} $, which gives $ A+BK=1.1 $.
	%an unstable closed loop system
%	\begin{equation*}
%	x(t+1)=x(t)+(0.1x(t)+1r(t))=1.1x(t)+r(t).
%	\end{equation*}
	We note that the closed-loop system is unstable.
	In other words,	only imposing a small ``distance" between the closed-loop system and the stable reference model does not guarantee the stability of the closed-loop system.
	%Theorem 2 allow unstable closed-loop systems to be part of the solution. 
	%We will therefore provide new necessary and sufficient conditions for approximate MRC with guaranteed stability.
\end{example}

\section{Including Stability Guarantees}

In this section, we will provide new necessary and sufficient conditions for approximate MRC with guaranteed stability.
Our objective is to find a condition under which there exists a controller that achieves both approximate MRC and stability.
This leads to the subsequent definition.
%Consider the set of systems that explain any set of input/state data $ (U_-,X) $.
%According to Definition \ref{D}, all these systems are contained in ellipsoids of systems centered around the reference model, that is, $ \Sigma_{(U_-,X)} \subseteq \Sigma^{K,L} $.
%Hence, an additional condition to guarantee the stability of all these systems should be that $ A+BK $ is Schur $ \forall (A,B) \in \Sigma^{K,L} $, which leads to the following definition.

\begin{definition}\label{DS} (Informativity for approximate MRC with guaranteed stability).
	Let $ A_{\rm m} $, $ B_{\rm m} $, $ D^A $, $ D^B $, $ \Gamma^A $, $ \Gamma^B $ be the matrices in \eqref{r}, \eqref{a}-\eqref{b}.
	The data $ (U_-,X) $, generated by system (\ref{1}) with noise model (\ref{noise}), are called informative for approximate MRC with guaranteed stability if there exist gains $ K $ and $ L $ such that $ \Sigma_{(U_-,X)} \subseteq \Sigma^{K,L} $ and for any $ (A,B) \in \Sigma^{K,L} $, $ A+BK $ is Schur.
\end{definition}

To formulate necessary and sufficient condition for the notion of informativity in Definition \ref{DS}, we first state the following technical lemma.

\begin{lemma}\label{LMI}
	($ \!\! $\cite[Prop. 5]{van2022data}).
	Let $R=\begin{bmatrix}
	R_{11}&R_{12}\\R_{21}&R_{22}
	\end{bmatrix}\in\mathbb{S}^{2n}$ and define the mapping
	\begin{equation}\label{psi}
	\Psi(\lambda):=R_{11}+R_{22}+\lambda R_{12}+\lambda^{-1} R_{21}.
	\end{equation}
	Then, there exists a $ P\in\mathbb{S}^{n} $ such that
	\begin{equation}\label{PLMI}
	\begin{bmatrix}
	P&0\\0&-P
	\end{bmatrix}-R>0
	\end{equation}
	if and only if $ \Psi(1)<0 $ and 
	\begin{equation}\label{e}
	\begin{bmatrix}
	0 & \Psi(1)^{-1} \\ \Psi(-1) & 2(R_{12}-R_{21})\Psi(1)^{-1}
	\end{bmatrix}
	\end{equation}
	 has no eigenvalues on the imaginary axis. 
\end{lemma}

The following theorem uses Lemma \ref{LZ} and Lemma \ref{LMI} to give a necessary and sufficient condition for stability of all systems in $ \Sigma^{K,L} $.

\begin{theorem}\label{TI}
	Define \begin{equation}\label{R}
	R=\begin{bmatrix}
	R_{11}&R_{12}\\R_{21}&R_{22}
	\end{bmatrix}:=\begin{bmatrix}
	D^A-A_{\rm m}\Gamma^AA_{\rm m}^{\top} & A_{\rm m}\Gamma^A \\ \Gamma^AA_{\rm m}^{\top} & -\Gamma^A
	\end{bmatrix}.
	\end{equation}
	Given $ \Psi(\lambda) $ in \eqref{psi}, we assume that $ \Psi(1) $ is invertible and \eqref{e} has no eigenvalues on the imaginary axis. 
	Then, given $ K $ and $ L $ such that $ \Sigma^{K,L} \neq \emptyset $, $ A+BK $ is Schur $ \forall (A,B) \in \Sigma^{K,L} $ if and only if  
	\begin{equation}\label{TS}
	A_{\rm m}^{\top}-I\in\mathcal{Z}_n^+\left(\begin{bmatrix}
	-D^A & 0 \\ 0 & \Gamma^A
	\end{bmatrix}\right).
	\end{equation}
\end{theorem}

\begin{proof}
	(\underline{Sufficiency}):
	Suppose (\ref{TS}) holds.
	Then, 
	\begin{equation}\label{PS1}
	\Psi(1)=\begin{bmatrix}
	I \\ A_{\rm m}^\top-I
	\end{bmatrix}^\top
	\begin{bmatrix}
	D^A & 0 \\ 0 & -\Gamma^A
	\end{bmatrix}
	\begin{bmatrix}
	I \\ A_{\rm m}^\top-I
	\end{bmatrix}<0.
	\end{equation}
	As (\ref{e}) has no eigenvalues on the imaginary axis, we have from Lemma \ref{LMI} that there exists a $ P\in\mathbb{S}^{n} $ such that	\eqref{PLMI} holds, implying that $ \forall (A,B) \in \Sigma^{K,L} $, 
	$$
	\begin{bmatrix}
	I \\ (A+BK)^\top
	\end{bmatrix}^\top
	\left(\begin{bmatrix}
	P&0\\0&-P
	\end{bmatrix}-R\right)
	\begin{bmatrix}
	I \\ (A+BK)^\top
	\end{bmatrix}>0.
	$$
	On the other hand, $ \forall (A,B) \in \Sigma^{K,L} $, we have
	\begin{align*}
	&\begin{bmatrix}
	I \\ (A+BK)^\top
	\end{bmatrix}^\top
	R
	\begin{bmatrix}
	I \\ (A+BK)^\top
	\end{bmatrix}\\
	=&D^A-(A+BK-A_{\rm m})\Gamma^A(A+BK-A_{\rm m})^\top\geq0.
	\end{align*}
	Then, for given $ P $ in (\ref{PLMI}), we have that $ \forall (A,B) \in \Sigma^{K,L} $, 
	\begin{equation}\label{PS}
	\begin{bmatrix}
	I \\ (A+BK)^\top
	\end{bmatrix}^\top
	\begin{bmatrix}
	P&0\\0&-P
	\end{bmatrix}
	\begin{bmatrix}
	I \\ (A+BK)^\top
	\end{bmatrix}>0.
	\end{equation}
	Next, we prove $ P>0 $.
	Choosing any $ B_0\in \mathbb{R}^{n\times m} $ such that
	$
	D^B-(B_0L-B_{\rm m})\Gamma^B(B_0L-B_{\rm m})^\top \geq 0,
	$
	and any $ A_0 $ such that $ A_0+B_0K=A_{\rm m} $, we have $ (A_0,B_0) \in \Sigma^{K,L} $.
	Then,
	\begin{equation}\label{P0}
	\begin{aligned}
	&\begin{bmatrix}
	I \\ (A_0+B_0K)^\top
	\end{bmatrix}^\top
	\begin{bmatrix}
	P&0\\0&-P
	\end{bmatrix}
	\begin{bmatrix}
	I \\ (A_0+B_0K)^\top
	\end{bmatrix}\\
	=&\begin{bmatrix}
	I \\ A_m^\top
	\end{bmatrix}^\top
	\begin{bmatrix}
	P&0\\0&-P
	\end{bmatrix}
	\begin{bmatrix}
	I \\ A_m^\top
	\end{bmatrix}>0,
	\end{aligned}
	\end{equation}
	that is, $ P $ satisfies the Lyapunov equation $ P-A_{\rm m}PA_{\rm m}^\top = Y $ for some matrix $ Y > 0 $.
	Since $ A_{\rm m} $ is Schur, the solution $ P $ is unique and given by $ P = \sum_{k=0}^{\infty}A_{\rm m}^kY(A_{\rm m}^\top)^k $.
	We observe that $ P > 0 $ because $ Y > 0 $. 
	Then, there exists a $ P > 0 $, such that $ \forall (A,B) \in \Sigma^{K,L} $, (\ref{PS}) holds, implying that $ A+BK $ is Schur $ \forall (A,B) \in \Sigma^{K,L} $.
	
	(\underline{Necessity}):
	Assume that $ A+BK $ is Schur $ \forall (A,B) \in \Sigma^{K,L} $.
	We want to prove that (\ref{TS}) holds.
	Suppose by contradiction that (\ref{TS}) does not hold, that is, 
	$$ \begin{bmatrix}
	I \\ A_{\rm m}^\top-I
	\end{bmatrix}^{\top}
	\begin{bmatrix}
	-D^A & 0 \\ 0 & \Gamma^A
	\end{bmatrix}
	\begin{bmatrix}
	I \\ A_{\rm m}^\top-I
	\end{bmatrix} $$
	has a non-positive eigenvalue.
	Then, there exists a nonzero vector $ x \in \mathbb{R}^n $ such that
	$$
	-\begin{bmatrix}
	x \\ x
	\end{bmatrix}^{\top}
	\begin{bmatrix}
	I & 0\\ A_{\rm m}^\top & -I
	\end{bmatrix}^{\top}
	\begin{bmatrix}
	-D^A & 0 \\ 0 & \Gamma^A
	\end{bmatrix}
	\begin{bmatrix}
	I & 0\\ A_{\rm m}^\top & -I
	\end{bmatrix}
	\begin{bmatrix}
	x \\ x
	\end{bmatrix}\geq 0.
	$$
	It can be verified that 
%	\begin{equation*}\label{key}
%	\begin{bmatrix}
%	I & 0\\ \!A_{\rm m}^{\top}\! & \!-I\!\!\;
%	\end{bmatrix}^{\!\!\!\top}\!\!\!
%	\begin{bmatrix}
%	\!D^A\! & \!0 \\ \!0 & \!\!-\Gamma^A\!
%	\end{bmatrix}\!\!
%	\begin{bmatrix}
%	I & 0\\ \!A_{\rm m}^\top\! & \!-I\!\!\;
%	\end{bmatrix}\!\!
%	=\!\!
%	\begin{bmatrix}
%	\!D^A\!\!\!\;-\!\!A_{\rm m}\Gamma^AA_{\rm m}^{\top}\! & \!\!\;A_{\rm m}\Gamma^A\!\!\; \\ \Gamma^AA_{\rm m}^{\top} & -\Gamma^A
%	\end{bmatrix}\!\!\in\! \mathbf{\Pi}_{n,n}.
%	\end{equation*}
	\begin{align*}
	-&\begin{bmatrix}
	I & 0\\ A_m^\top & -I
	\end{bmatrix}^{\top}
	\begin{bmatrix}
	-D^A & 0 \\ 0 & \Gamma^A
	\end{bmatrix}
	\begin{bmatrix}
	I & 0\\ A_m^\top & -I
	\end{bmatrix}\\
	=&\begin{bmatrix}
	D^A-A_m\Gamma^AA_m^{\top} & A_m\Gamma^A \\ \Gamma^AA_m^{\top} & -\Gamma^A
	\end{bmatrix}\in \mathbf{\Pi}_{n,n}.
	\end{align*}
	Taking any $ B_0\in \mathbb{R}^{n\times m} $ such that
	\begin{equation*}
	D^B-(B_0L-B_{\rm m})\Gamma^B(B_0L-B_{\rm m})^\top \geq 0,
	\end{equation*}
	we obtain from Lemma \ref{LZ} that there exists $ A_0\in \mathbb{R}^{n\times n}: $
	\begin{equation}\label{P1}
	(A_0+B_0K)^{\top} \in \mathcal{Z}_{r}\left(\begin{bmatrix}
	D^A-A_{\rm m}\Gamma^AA_{\rm m}^{\top} & A_{\rm m}\Gamma^A \\ \Gamma^AA_{\rm m}^{\top} & -\Gamma^A
	\end{bmatrix}\right),
	\end{equation}
	such that
	\begin{equation}\label{P2}
	\begin{bmatrix}
	x\\x
	\end{bmatrix}= \begin{bmatrix}
	I\\ (A_0+B_0K)^\top
	\end{bmatrix}x.
	\end{equation}
	Obviously, (\ref{P1}) is equivalent to 
	\begin{equation}\label{key}
	D^A-(A_0+B_0K-A_{\rm m})\Gamma^A(A_0+B_0K-A_{\rm m})^\top \geq 0,
	\end{equation}
	that is, $ (A_0,B_0) \in \Sigma^{K,L} $.
	However, (\ref{P2}) implies that $ A_0+B_0K $ has an eigenvalue $ 1 $ and is thus not Schur, resulting in a contradiction.
	We thus conclude that (\ref{TS}) holds if $ A+BK $ is Schur $ \forall (A,B) \in \Sigma^{K,L} $.
\end{proof}

%It should be noticed that the fact that \eqref{TS} does not depend on $ K $ and $ L $ implies that Theorem \ref{TI} holds for any $ K $ and $ L $, indicating that \eqref{TS} is a necessary and sufficient condition for the system states $ x $ to be bounded.
%Hence, we can conclude that with $ D^B\!\!>\!0 $ and an eigenvalue assumption that can be directly checked given $ A_{\rm m},D^A,\Gamma^A $, the conditions that $ (A_{\rm s},B_{\rm s})\in\Sigma^{K,L} $ and \eqref{TS} holds are necessary and sufficient for system \eqref{1} to approximately match the reference model \eqref{r} with guaranteed stability.
%Combining Theorem \ref{TDIAMRC} and Theorem \ref{TI}, we obtain 
It is noteworthy that \eqref{TS} is independent of the particular gains $ K $ and $ L $.
This is instrumental in deriving the following necessary and sufficient conditions for data-driven approximate MRC with guaranteed stability.
\begin{theorem}\label{TDIAMRCSG}
	Let $ \Phi \in \mathbf{\Pi}_{n,T} $.
	Assume that $ (i)\ N \not\leq 0 $ and \eqref{e} has no eigenvalues on the imaginary axis; or $ (ii)\ D^A=D^B=0 $.
	Then, the data $ (U_-,X) $, generated by system (\ref{1}) with noise model (\ref{noise}), are informative for approximate MRC with guaranteed stability if and only if there exist $ K,L $ and $ \alpha_1,\alpha_2>0 $ such that \eqref{TK}, \eqref{TL} and \eqref{TS} hold.
\end{theorem}
\begin{proof}
	In case $ (i) $, the assertion of Theorem \ref{TDIAMRCSG} is obtained by directly combining Theorems \ref{TDIAMRC} and \ref{TI},  since $ (i) $ meets the assumptions of both Theorem \ref{TDIAMRC} and \ref{TI}.
	In case $ (ii) $, we first obtain from \eqref{PMK} that $ \forall (A,B) \in \Sigma^{K,L} $,
	\begin{equation*}
	x(t+1)=Ax(t)+B(Kx(t)+Lr(t))=A_{\rm m}x(t)+B_{\rm m}r(t),
	\end{equation*}
	which is a stable closed-loop system as $ A_{\rm m} $ is Schur.
	On the other hand, \eqref{TS} always holds when $ A_{\rm m} $ is Schur and $ D^A=0 $.
	Hence, in case $ (ii) $, the assertion of Theorem \ref{TDIAMRCSG} is derived directly from Theorem \ref{TDIAMRC}.
\end{proof}
\begin{remark} \label{RB}
	Theorem \ref{TDIAMRCSG} is a general result that also covers the noiseless scenario by setting $ D^A=D^B=0 $, $ \Gamma^A=I_n$, $ \Gamma^B=I_p $ and $ \Phi=\begin{bmatrix}
	0&0\\0&-I
	\end{bmatrix} \in \mathbf{\Pi}_{n,T} $.
	With these special choices, $ A_{\rm m} $ being Schur implies that \eqref{TS} holds, and it can be shown that \eqref{TK} and \eqref{TL} are equivalent to \eqref{T1K} and \eqref{T1L}, that is, informativity for approximate MRC with guaranteed stability is equivalent to informativity for MRC, and the conditions of Theorem \ref{TDIAMRCSG} boil down to those of Corollary \ref{CDI}. 
%	$ \hfill\square $
\end{remark}

\section{Numerical Example}

This section illustrates the theoretical results, especially Theorem \ref{TDIAMRCSG}, via a highly maneuverable aircraft model with unstable longitudinal dynamics (angle of attack, pitch rate, and pitch angle) \cite{hartmann1979control,yuan2018robust}. After discretization with sampling time of 0.01, we get a system model in the form \eqref{1}, with
\begin{align*}
A_{\rm s} &=
\begin{bmatrix}
0.9810 & 0.0098 & 0 \\
0.1172 & 0.9737 & 0 \\
0 & 0.01 & 1\\
\end{bmatrix},\\
B_{\rm s} & =
\begin{bmatrix}
-0.0024 & -0.0017 & 0 & -0.0020 \\
-0.4621 & -0.3160 & 0.2240 & -0.3118 \\
0 & 0 & 0 & 0 \\
\end{bmatrix},
\end{align*}
and a reference model in the form \eqref{r}, with
\begin{equation*}
A_{\rm m} = \begin{bmatrix}
0.9800 & 0.0065 & -0.0075 \\
-0.0767 & 0.2964 & -1.5178 \\
0 & 0.01 & 1 \\
\end{bmatrix}, B_{\rm m}=B_{\rm s}.
\end{equation*}

% The three states of the system represent angle of attack, pitch rate, and pitch angle; the four inputs represent elevator, elevon, canard, and symmetric aileron. 
Because the system is open-loop unstable, we consider closed-loop experiments of length $ T = 100 $, with an internally stabilizing controller $ u=K_0x+L_0r $,
\begin{align*}
	K_0&=\begin{bmatrix}
	0.7477 & -0.0511 & 0.4806 \\
	0.7160 & 0.3976 & -0.0423 \\
	0.2418 & -0.1610 & -0.8422 \\
	0.3790 & -0.5398 & -0.4469 \\
	\end{bmatrix},\\
	L_0&=\begin{bmatrix}
	0.7109 & -0.2894 & 0.6691 & 0.1629 \\
	-0.5317 & 0.6292 & 0.5491 & 0.5267 \\
	0.5001 & 0.7104 & -0.1839 & 0.4843 \\
	0.2632 & 0.5693 & 0.1066 & -0.7163 \\
	\end{bmatrix}
\end{align*}
with $ r $ drawn from the standard normal distribution, and random initial conditions.
To generate noisy data, the noise follows the considerations in \cite{hartmann1979control,yuan2018robust}, namely, the noise affecting the second state equation is roughly two orders of magnitude larger than the noise affecting the first state equation, while the third state equation is noiseless because it represents integration of the pitch rate.
% are in different levels, and there is no noise present in the $ 3^{rd} $ component since it represents a straightforward derivative relationship.
%Moreover, to discretize the system, the dynamics were multiplied by 0.01, which implies that the effect of the input $ u $ on the dynamics is also scaled by a factor of 0.01.
%Hence, it is reasonable to assume that the same scaling applies to the noise signal as well.
Given a noise level $ w_l $ (a general ratio of noise to state amplitude),
we take the noise model in the form of \eqref{noise}, with
\begin{equation*}
\Phi_{11}=\begin{bmatrix}
0.001w_l^2 & 0 & 0 \\
0 & 10w_l^2 & 0 \\
0 & 0 & 0 \\
\end{bmatrix}, \Phi_{12}= \Phi_{21}^\top=0, \Phi_{22}=-I.
\end{equation*}
In particular, we choose 31 different noise levels for $ w_l $ and generate 200 noisy datasets for each noise level. 
The objective is to find the controller \eqref{controller} achieving data-driven approximate MRC with guaranteed stability.
Instead of pre-defining the difference matrices $ D^A $ and $ D^B $, we take $ \Gamma^A=I, \Gamma^B=I $ and minimize tr$ (D^A) +\text{tr}(D^B) $ under the LMI constraints \eqref{TK}, \eqref{TL} and \eqref{TS} as discussed in Remark 1.

Let us first consider $ 21 $ noise levels $ w_l\in\{0,0.05,\dots,1\} $.
%For each dataset, the eigenvalue condition of \eqref{e} is satisfied using the obtained $ D^A $.
% the $ 11^{th} $ largest and smallest 
For each level, we report in Fig. \ref{DADB} the average and 90\% confidence intervals of the traces of the resulting $ D^A $ and $ D^B $.
We also verify that the eigenvalue condition of \eqref{e} is always satisfied with the obtained $ D^A $.
%Fig. \ref{DADB} illustrates the trend of the trace of $ D^A $ and $ D^B $ as the noise level $ w_l $ increases, where the lines show the average trace and the shaded area contains $ 90\% $ of the traces. 
\begin{figure}[htbp]
	\centering
	\includegraphics[trim=49bp 339bp 61bp 350bp, clip, width=8cm]{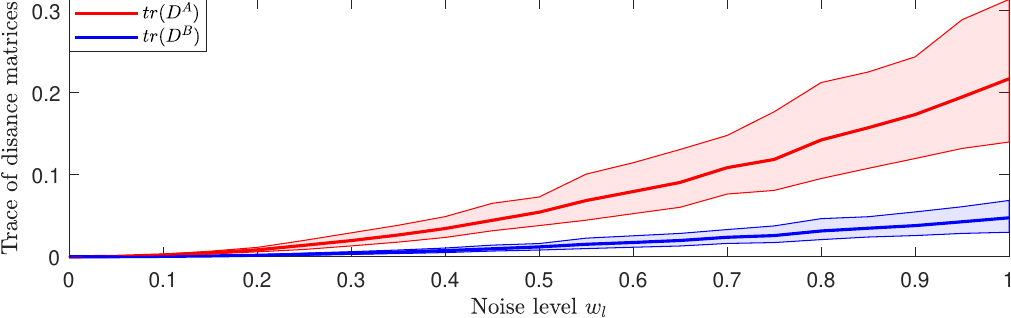}
	\caption{Traces of $ D^A $ and $ D^B $ as the noise level increases.}
	\label{DADB}
\end{figure}\\
As the noise level $ w_l $ increases beyond 1, the probability of obtaining an approximate MRC controller with guaranteed stability continues to decrease, as illustrated in the table below.
%the MRC controller with guaranteed stability cannot be found for the first time at the level $ w_l=1.2 $.
%The table below illustrates the probability of obtaining an approximate MRC controller with guaranteed stability across different levels of noise.
Once the noise level surpasses $ w_l=2 $, the control objective can hardly be achieved.
\begin{table}[htbp]
	\centering
	\label{tab:simple_table}
	\begin{tabular}{|c|c|c|c|c|}
		\hline
		$ w_l=1.1 $ & $ w_l=1.2 $ & $ w_l=1.3 $ & $ w_l=1.4 $ & $ w_l=1.5 $ \\
		\hline
		99\% & 98\% & 87.5\% & 71\% & 54\% \\
		\hline
		$ w_l=1.6 $ & $ w_l=1.7 $ & $ w_l=1.8 $ & $ w_l=1.9 $ & $ w_l=2.0 $ \\
		\hline
		35.5\% & 15.5\% & 11\% & 4\% & 1.5\%\\
		\hline
	\end{tabular}
\end{table}\\
Finally, for three datasets in noise levels $ w_l\in\{0,0.1,1\} $, we provide in Fig. \ref{error} the trajectories of $e=x-x_{{\rm m}}$ resulting from the obtained control gains $ K $ and $ L $,
% and the noise drawn from the standard normal distribution in Fig. \ref{noisedata}, 
from which stability and good regulation can be observed.
%Given the same random reference inputs, we show the tracking errors between the noisy closed-loop system with the controller \eqref{controller} and the dynamics of \eqref{r} in Fig. \ref{error}.
%In these cases, the eigenvalue of \eqref{e} is $ \{\pm121.1618, \pm74.1695, \pm1.9388\} $.
%\begin{figure}[htbp]
%	\centering
%	\includegraphics[trim=100bp 260bp 110bp 270bp, clip, height=3.2cm]{DADBcl.pdf}
%	\caption{Traces of $ D^A $ and $ D^B $ as the noise level increases.}
%	\label{DADBcl}
%\end{figure}

\begin{figure}[htbp]
	\centering
	\includegraphics[trim=10bp 204bp 10bp 225bp, clip, width=8cm]{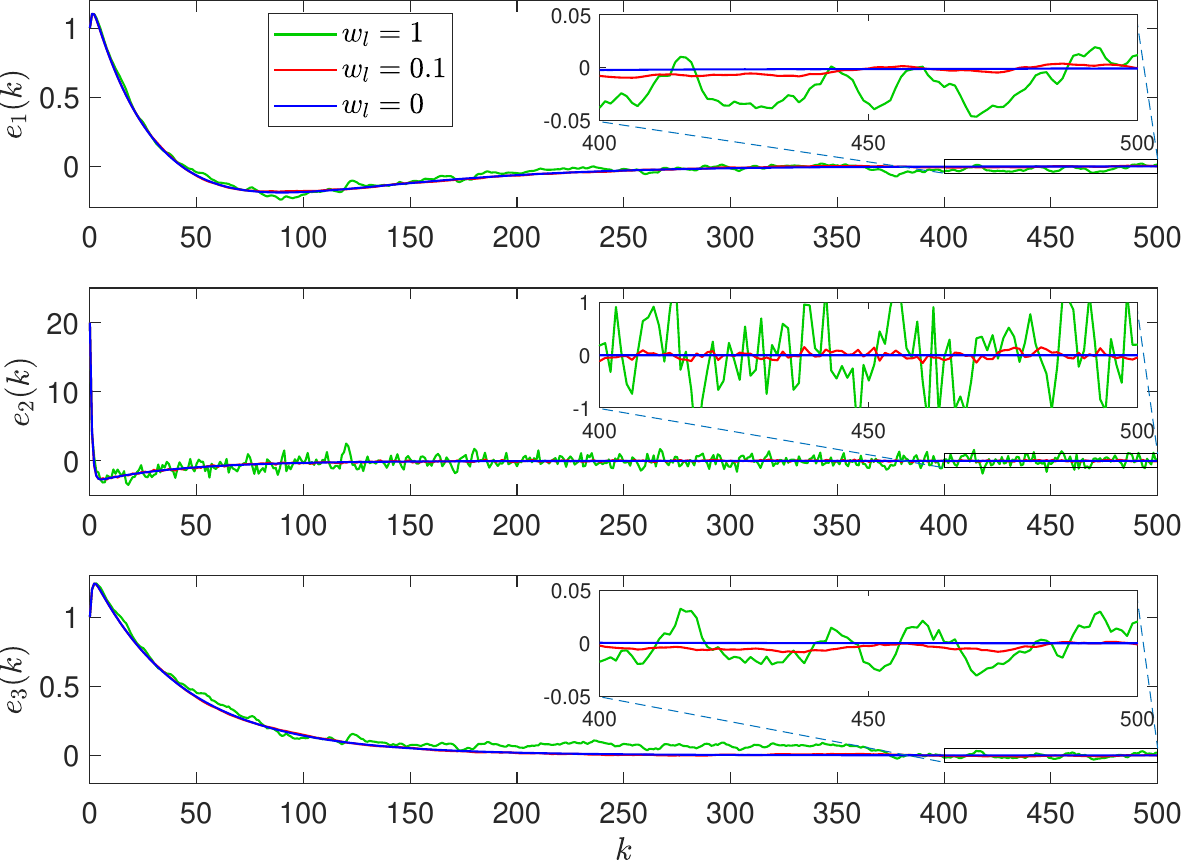}
	\caption{State tracking error with the noise levels $ 0,0.1,1 $.}
	\label{error}
\end{figure}

\section{Conclusions}
In this paper, the data-driven MRC problem has been studied within the framework of data informativity, and necessary and sufficient conditions for data have been developed to achieve MRC in noise-free and noisy settings.
In the noise-free setting, we have determined feedback gains so that all closed-loop systems consistent with the data match the reference model exactly.
In the noisy data setting, we have defined ellipsoids of systems centered around the reference model, and determined feedback gains for data-driven approximate MRC so that all closed-loop systems consistent with the data are within these ellipsoids.
In our pursuit of providing stability guarantees, we have found an interesting condition for ensuring stability across all closed-loop systems, which surprisingly remains independent of the specific feedback gains.
This was key in the derivation of necessary and sufficient conditions for approximate MRC with guaranteed stability.
Possible future research could involve exploring the scenario where input and output data are available.

\bibliographystyle{IEEEtran}
\bibliography{References}

% Generated by IEEEtran.bst, version: 1.13 (2008/09/30)
\begin{thebibliography}{10}
\providecommand{\url}[1]{#1}
\csname url@samestyle\endcsname
\providecommand{\newblock}{\relax}
\providecommand{\bibinfo}[2]{#2}
\providecommand{\BIBentrySTDinterwordspacing}{\spaceskip=0pt\relax}
\providecommand{\BIBentryALTinterwordstretchfactor}{4}
\providecommand{\BIBentryALTinterwordspacing}{\spaceskip=\fontdimen2\font plus
\BIBentryALTinterwordstretchfactor\fontdimen3\font minus
  \fontdimen4\font\relax}
\providecommand{\BIBforeignlanguage}[2]{{%
\expandafter\ifx\csname l@#1\endcsname\relax
\typeout{** WARNING: IEEEtran.bst: No hyphenation pattern has been}%
\typeout{** loaded for the language `#1'. Using the pattern for}%
\typeout{** the default language instead.}%
\else
\language=\csname l@#1\endcsname
\fi
#2}}
\providecommand{\BIBdecl}{\relax}
\BIBdecl

\bibitem{ioannou2006adaptive}
P.~Ioannou and B.~Fidan, \emph{Adaptive control tutorial}.\hskip 1em plus 0.5em
  minus 0.4em\relax SIAM, 2006.

\bibitem{lavretsky2009combined}
E.~Lavretsky, ``Combined/composite model reference adaptive control,''
  \emph{IEEE Transactions on Automatic Control}, vol.~54, no.~11, pp.
  2692--2697, 2009.

\bibitem{chowdhary2010concurrent}
G.~Chowdhary and E.~Johnson, ``Concurrent learning for convergence in adaptive
  control without persistency of excitation,'' in \emph{49th IEEE Conference on
  Decision and Control (CDC)}, 2010, pp. 3674--3679.

\bibitem{roy2017combined}
S.~B. Roy, S.~Bhasin, and I.~N. Kar, ``Combined {MRAC} for unknown {MIMO LTI}
  systems with parameter convergence,'' \emph{IEEE Transactions on Automatic
  Control}, vol.~63, no.~1, pp. 283--290, 2017.

\bibitem{lee2019concurrent}
H.-I. Lee, H.-S. Shin, and A.~Tsourdos, ``Concurrent learning adaptive control
  with directional forgetting,'' \emph{IEEE Transactions on Automatic Control},
  vol.~64, no.~12, pp. 5164--5170, 2019.

\bibitem{hou2013model}
Z.-S. Hou and Z.~Wang, ``From model-based control to data-driven control:
  {Survey}, classification and perspective,'' \emph{Information Sciences}, vol.
  235, pp. 3--35, 2013.

\bibitem{willems2005note}
J.~C. Willems, P.~Rapisarda, I.~Markovsky, and B.~L. De~Moor, ``A note on
  persistency of excitation,'' \emph{Systems \& Control Letters}, vol.~54,
  no.~4, pp. 325--329, 2005.

\bibitem{breschi2021direct}
V.~Breschi, C.~De~Persis, S.~Formentin, and P.~Tesi, ``Direct data-driven
  model-reference control with {Lyapunov} stability guarantees,'' in \emph{2021
  60th IEEE Conference on Decision and Control (CDC)}, 2021, pp. 1456--1461.

\bibitem{de2019formulas}
C.~De~Persis and P.~Tesi, ``Formulas for data-driven control: {Stabilization},
  optimality, and robustness,'' \emph{IEEE Transactions on Automatic Control},
  vol.~65, no.~3, pp. 909--924, 2019.

\bibitem{pan2022stochastic}
G.~Pan, R.~Ou, and T.~Faulwasser, ``On a stochastic fundamental lemma and its
  use for data-driven optimal control,'' \emph{IEEE Transactions on Automatic
  Control}, vol.~68, no.~10, pp. 5922--5937, 2023.

\bibitem{schmitz2022willems}
P.~Schmitz, T.~Faulwasser, and K.~Worthmann, ``Willems' fundamental lemma for
  linear descriptor systems and its use for data-driven output-feedback
  {MPC},'' \emph{IEEE Control Systems Letters}, vol.~6, pp. 2443--2448, 2022.

\bibitem{van2020data}
H.~J. van Waarde, J.~Eising, H.~L. Trentelman, and M.~K. Camlibel, ``Data
  informativity: a new perspective on data-driven analysis and control,''
  \emph{IEEE Transactions on Automatic Control}, vol.~65, no.~11, pp.
  4753--4768, 2020.

\bibitem{van2020noisy}
H.~J. van Waarde, M.~K. Camlibel, and M.~Mesbahi, ``From noisy data to feedback
  controllers: {Nonconservative} design via a matrix {S-lemma},'' \emph{IEEE
  Transactions on Automatic Control}, vol.~67, no.~1, pp. 162--175, 2020.

\bibitem{polik2007survey}
I.~P{\'o}lik and T.~Terlaky, ``A survey of the {S-lemma},'' \emph{SIAM Review},
  vol.~49, no.~3, pp. 371--418, 2007.

\bibitem{van2023quadratic}
H.~J. van Waarde, M.~K. Camlibel, J.~Eising, and H.~L. Trentelman, ``Quadratic
  matrix inequalities with applications to data-based control,'' \emph{SIAM
  Journal on Control and Optimization}, vol.~61, no.~4, pp. 2251--2281, 2023.

\bibitem{van2023informativity}
H.~J. Van~Waarde, J.~Eising, M.~K. Camlibel, and H.~L. Trentelman, ``The
  informativity approach: To data-driven analysis and control,'' \emph{IEEE
  Control Systems Magazine}, vol.~43, no.~6, pp. 32--66, 2023.

\bibitem{van2022data}
H.~J. van Waarde, M.~K. Camlibel, and H.~L. Trentelman, ``Data-driven analysis
  and design beyond common {Lyapunov} functions,'' in \emph{2022 IEEE 61st
  Conference on Decision and Control (CDC)}, 2022, pp. 2783--2788.

\bibitem{hartmann1979control}
G.~Hartmann, M.~Barrett, and C.~Greene, ``Control design for an unstable
  vehicle,'' \emph{NASA CR-170393}, 1979.

\bibitem{yuan2018robust}
S.~Yuan, B.~De~Schutter, and S.~Baldi, ``Robust adaptive tracking control of
  uncertain slowly switched linear systems,'' \emph{Nonlinear Analysis: Hybrid
  Systems}, vol.~27, pp. 1--12, 2018.

\end{thebibliography}

\end{document}